\documentclass{amsart}
\usepackage{graphicx}
\usepackage[all]{xy}
\usepackage[ansinew]{inputenc}
\usepackage{amsmath}
\usepackage{amscd}
\usepackage{amssymb}
\usepackage{latexsym}

\newtheorem{thm}{Theorem}[section]

\theoremstyle{definition}
\newtheorem{cor}[thm]{Corollary}
\newtheorem{lem}[thm]{Lemma}

\newtheorem{defn}[thm]{Definition}

\newtheorem{fact}[thm]{Fact}

\newtheorem*{thmA}{Theorem A}
\newtheorem*{thmB}{Theorem B}
\newtheorem*{thmC}{Theorem C}
\newtheorem*{thmD}{Theorem D}

\numberwithin{equation}{section}

\newcommand{\N}{\mathbb{N}}
\newcommand{\Z}{\mathbb{Z}}
\newcommand{\Q}{\mathbb{Q}}
\newcommand{\R}{\mathbb{R}}

\newcommand{\Cal}{\mathcal}

\newcommand{\Odd}{\operatorname{Odd}}
\newcommand{\Even}{\operatorname{Even}}

\def \<{\langle}
\def \>{\rangle}

\def \((  {(\!(}
\def \)) {)\!)}

\begin{document}

\title[]
{Expansions of the ordered additive group of real numbers by two discrete subgroups}

\author[P. Hieronymi]{Philipp Hieronymi}
\address
{Department of Mathematics\\University of Illinois at Urbana-Champaign\\1409 West Green Street\\Urbana, IL 61801}
\email{phierony@illinois.edu}
\urladdr{http://www.math.uiuc.edu/\textasciitilde phierony}

\thanks{The author was partially supported by NSF grant DMS-1300402 and by UIUC Campus Research Board award 13086.}
\date{\today}

\begin{abstract}
The theory of $(\R,<,+,\Z,\Z a)$ is decidable if $a$ is quadratic. If $a$ is the golden ratio, $(\R,<,+,\Z,\Z a)$ defines multiplication by $a$. The results are established by using the Ostrowski numeration system based on the continued fraction expansion of $a$ to define the above structures in monadic second order logic of one successor. The converse that $(\R,<,+,\Z,\Z a)$ defines monadic second order logic of one successor, will also be established.
\end{abstract}
\maketitle

\section{Introduction}

Let $a\in \R$. We consider the following structure $\Cal R_a := (\R,<,+,\Z,\Z a)$. Although it is well known that $(\R,<,+,\Z)$ has a decidable theory and other desirable model theoretic properties (arguably due to Skolem \cite{skolem}\footnote{Skolem essentially showed elimination of quantifiers ranging over elements of $\Z$. Full quantifier elimination follows easily as pointed out by Smorny\'nski \cite[Exercise III.4.15]{smor}.} and later rediscovered independently by Weispfenning \cite{weis} and Miller \cite{ivp}), the question whether the theory of $\Cal R_a$ is decidable even for some irrational number $a$ has been open for a long time. The interest in these structures arises among other things from the observation that the structure $\Cal R_a$ codes many of the Diophantine properties of $a$. This observation will play a key role throughout this paper.  The following is the main result of this paper.

\begin{thmA} If $a$ is quadratic, then the theory of $\Cal R_{a}$ is decidable.
\end{thmA}

A real number is called quadratic if it is the solution to a quadratic equation with rational coefficients. Theorem A provides the first example of an irrational number $a$ such that the theory of $\Cal R_a$ is decidable. Its proof depends crucially on the periodicity of the continued fraction expansion of $a$. When $a$ is non-quadratic, the conclusion of Theorem A can fail. It will be shown that whenever the continued fraction expansion of $a$ is non-computable, then the theory of $\Cal R_a$ is undecidable.  It is also worth noting that while the theory of $\Cal R_a$ can be decidable, its expansion $(\R,<,+,\Z,\Z a,\Z b)$ defines multiplication on $\R$ and hence its theory is undecidable as along as $1,a,b \in \R$ are linearly independent over $\Q$, by Hieronymi and Tychonievich \cite[Theorem C]{HT}.\\

Now consider the structure $\Cal S_a := (\R,<,+,\Z,\lambda_a)$, where $\lambda_a:\R \to \R$ maps $x$ to $ax$. Note that $\Cal S_a$ is an expansion of $\Cal R_a$, since $\lambda_a(\Z)=a\Z$. There are more results known about these structures than about $\Cal R_a$. If $a$ is not a quadratic real number, $\Cal S_a$ defines multiplication on $\R$ and hence its theory is undecidable by \cite[Theorem B]{HT}.  However until now there was no known example of an irrational number $a$ such that the theory of $\Cal S_a$ is decidable. The following Theorem gives the first example of such a real number.

\begin{thmB} Let $\varphi := \frac{1 + \sqrt{5}}{2}$ be the golden ratio. Then $\Cal R_{\varphi}$ defines $\lambda_{\varphi}$ and hence the theory of $\Cal S_{\varphi}$ is decidable.
\end{thmB}

Definable here and throughout the paper will always mean definable without parameters. In order to establish Theorem A, we will show that for $a$ quadratic, $\Cal R_{a}$ is definable in monadic second order logic of one successor. To make this statement precise, consider the two-sorted structure $\Cal B:=(\N,\Cal P(\N),s_{\N},\in)$, where $s_{\N}$ is the successor function on $\N$ and $\in$ is the relation on $\N \times \Cal P(\N)$ such that $\in(t,X)$ iff $t \in X$. The structure $\Cal B$ was studied by Büchi in his seminal paper \cite{Buchi}. Using the theory of automata Büchi proved that the theory of $\Cal B$ is decidable and established what would today be called a quantifier elimination result.
Theorem A will follow immediately from the decidability of the theory of $\Cal B$ and the following result.

\begin{thmC} Let $a \in \R$ be quadratic. Then $\Cal B$ defines an isomorphic copy of $\Cal R_{a}$.
\end{thmC}

A structure that is isomorphic to a definable structure in $\Cal B$ is sometimes called B\"uchi presentable. While Theorem C shows that $\Cal R_a$ is at most as complicated as $\Cal B$ for quadratic $a$, we will also establish the converse.

\begin{thmD} Let $a \in \R$ be irrational. Then $\Cal R_{a}$ defines an isomorphic copy of $\Cal B$.
\end{thmD}

One can show that $(\R,<,+,\Z)$ does not define an isomorphic copy of $\Cal B$ and is significantly less complicated than $\Cal B$. Hence Theorem D shows that while the theory of $\Cal R_a$ can be decidable, $\Cal R_a$ is not as well-behaved as $(\R,<,+,\Z)$.\newline

To prove Theorems C and D, we will rely on results from the theory of Diophantine approximation. The key tool to construct the isomorphic copies in Theorem C and D will be the Ostrowski representations of both natural numbers and real numbers due to Ostrowski \cite{Ost}. These representations originating in the theory of Diophantine approximation are based on the continued fraction expansion of $a$. The reason why the construction in Theorem C works for quadratic numbers and not for others, is that a real number $a$ has a periodic continued fraction expansion if and only if $a$ is quadratic.\newline

This is not the first time that B\"uchi's Theorem is used to understand expansions of the ordered real additive group. As mentioned by Boigelot, Rassart and Wolper in \cite{BRW}, B\"uchi himself must have known that the the structure $(\R,<,+,\Z)$ is definable in $\Cal B$ and hence that its theory is decidable. Also in \cite{BRW}, B\"uchi's Theorem is used to show that the theory of the expansion of $(\R,<,+,\Z)$ by a ternary predicate $V_r(x,u,k)$ that holds iff $u$ is a positive integer power of $r$, $k\in \{0,\dots,r-1\}$ and the digit of the base-$r$ representation of $x$ in the position corresponding to $u$ is $k$, is decidable. In some sense, their use of base-$r$ representations will be replaced in this paper by the Ostrowski representations.\newline

The results of the paper should not only be of theoretical importance. The decidability of the theory $(\R,<,+,\Z)$ has been used in verification and model checking, since mixed real-integers constraints appear naturally there.
Hence the results of this paper should be relevant in this area, if only by showing that there are interesting expansions of $(\R,<,+,\Z)$ whose theory is decidable.

\subsection*{Acknowledgements} I would like to thank Chris Miller for bringing these questions to my attention several years ago, Lou van den Dries for comments on an early version of this paper and Christiane Frougny for pointing out references. I am grateful to the anonymous referee for closely reading the paper and helpful comments that improved the presentation of this paper.

\subsection*{Notation} We denote $\{0,1,2,\dots\}$ by $\N$. Throughout this paper \textbf{definable} will mean definable without parameters.

\section{Diophantine approximations}

In this section we will recall some definitions and results from the study of Diophantine approximations. For more details and proofs, see Rockett and Szüsz \cite{RS}.

\begin{defn} A fraction $p/q \in \Q$ is a \textbf{best rational approximation} of a real number $a$ if for every fraction $\frac{p'}{q'}$ with $1 \leq q' \leq q$ and $p/q \neq p'/q'$
\[
|q'a - p'| > |qa - p|
\]
\end{defn}
Note that using $|a -p/q|$ instead of $|qa-p|$ changes the definition. For that reason the approximations in previous definition are sometimes called best rational approximation of the second kind.

\begin{defn} A \textbf{continued fraction expansion} $[a_0;a_1,\dots,a_k,\dots]$ is an expression of the form
\[
a_0 + \frac{1}{a_1 + \frac{1}{a_2+ \frac{1}{a_3+  \frac{1}{\ddots}}}}
\]
For a real number $a$, we say $[a_0;a_1,\dots,a_k,\dots]$ is the \textbf{continued fraction expansion of $a$} if $a=[a_0;a_1,\dots,a_k,\cdots]$ and $a_0\in \Z$, $a_i\in \N_{>0}$ for $i>0$.
\end{defn}

\noindent It is well known that every real number has a unique continued fraction expansion. For the rest of this section, fix a real number $a$ with continued fraction expansion $[a_0;a_1,\dots,a_k,\dots]$.

\begin{fact}\cite[Chapter III.1 Theorem 1 and 2]{RS}\label{fact:quadraticperiodic} The continued fraction expansion of $a$ is periodic iff $a$ is a quadratic irrational.
\end{fact}

\begin{defn}\label{def:beta} Let $k\geq 1$. We define $p_k/q_k \in \Q$ to be the \textbf{$k$-th convergent of $a$}, that is
\[
\frac{p_k}{q_k} = [a_0;a_1,\dots,a_k].
\]
The \textbf{$k$-th difference of $a$} is defined as $\beta_k := q_k a - p_k$. \newline
We define $\zeta_k \in \R$ to be the \textbf{$k$-th complete quotient of $a$}, that is
\[
\zeta_k = [a_k;a_{k+1},a_{k+2},\dots].
\]
\end{defn}
\noindent It is worth pointing out that for $k>0$, $\zeta_k > 1$, since $a_k$ is positive.

\begin{fact}{\cite[Chapter II.2 Theorem 2]{RS}}\label{bestrational} The set of best rational approximations of $a$ is precisely the set of all convergents of $a$.
\end{fact}

\begin{fact}{\cite[Chapter I.1 p. 2]{RS}}\label{fact:recursive} Let $q_{-1} := 0$ and $p_{-1}:=1$. Then $q_{0} = 1$, $p_{0}=a_0$ and for $k\geq 0$,
\begin{align*}
q_{k+1} &= a_{k+1} \cdot q_k + q_{k-1}, \\
p_{k+1} &= a_{k+1} \cdot p_k + p_{k-1}. \\
\end{align*}
\end{fact}
\noindent It follows immediately that for $k\geq 0$, $\beta_{k+1} = a_{k+1} \beta_k + \beta_{k-1}$.

\begin{fact}{\cite[Chapter I.4  p. 9]{RS}} \label{fact:beta} Let $k\in \N_{>0}$. Then
\[
\beta_{k+1} = - \frac{\beta_k}{\zeta_{k+2}}.
\]
\end{fact}
\noindent Since $\zeta_k >1$, the absolute value of $\beta_k$ decreases with $k$.

\begin{fact}{\cite[Chapter II.4  p. 24]{RS}}\label{ostrowski} Let $N\in \N$. Then $N$ can be written uniquely as
\[
N = \sum_{k=0}^{n} b_{k+1} q_{k},
\]
where $b_k \in \N$ such that $b_1<a_1$, $b_k \leq a_{k}$ and, if $b_k = a_{k}$, $b_{k-1} = 0$.
\end{fact}

\noindent The representation in the previous fact is called the \textbf{Ostrowski representation} of $N$ based on $a$. This representation will play a crucial role later.
If $a$ is the golden ratio, the Ostrowski respresentation based on $a$ is better known as the Zeckendorf representation, see Zeckendorf \cite{Zeckendorf}. It is important to note that the Ostrowski representation is obtained by a greedy algorithm, see \cite[Chapter II.4  p. 24]{RS}. The following fact follows immediately.

\begin{fact}\label{fact:zorder} Let $M,N\in \N$ with $M\neq N$ and let $\sum_{k} b_k q_{k}$ and $\sum_{k} c_k q_{k}$ be the Ostrowski representation of $M$ and $N$. Let $n\in \N$ be the maximal such that $b_n \neq c_n$. Then
$M < N$ iff $b_n < c_n$.
\end{fact}

\noindent We will also need a similar representation of a real number.

\begin{fact}{\cite[Chapter II.6  Theorem 1]{RS}\footnote{While Fact \ref{ostrowskireal} is well known, the statement of \cite[Chapter II.6  Theorem 1]{RS} is unfortunately slightly different. But by inspection of the proof in \cite{RS} and using Fact \ref{fact:osteven} and Fact \ref{fact:ostodd} the reader can easily verify that Fact \ref{ostrowskireal} indeed follows from the statement in \cite{RS}.}}
\label{ostrowskireal} Let $c \in \R$ be such that $-\frac{1}{\zeta_1} \leq c < 1-\frac{1}{\zeta_1}$. Then $c$ can be written uniquely in the form
\[
c = \sum_{k=0}^{\infty} b_{k+1} \beta_{k},
\]
where $b_k \in \N$, $0 \leq b_1 < a_1$, $0 \leq b_k \leq a_{k}$, for $k> 1$, and $b_k = 0$ if $b_{k+1} = a_{k+1}$, and $b_k < a_{k}$ for infinitely many odd $k$.
\end{fact}

\noindent  One property that is used in the proof of Fact \ref{ostrowskireal} is of particular importance to us.

\begin{fact}{\cite[Chapter II.6  p.32f]{RS}} \label{fact:osteven} Let $n \in \N$. Then
\[
- \beta_n = a_{n+2} \beta_{n+1} + a_{n+4} \beta_{n+3} + a_{n+6} \beta_{n+5} + \dots
\]
\end{fact}

\noindent Hence if $n$ is even, this equation determines the Ostrowski representation of $-\beta_n$.

\begin{fact}\label{fact:ostodd} Let $n\in \N$ be odd. Then the Ostrowski representation of $-\beta_n$ is
\[
\beta_{n-1}+(a_{n+1}-1) \beta_{n} + a_{n+3} \beta_{n+2} + a_{n+5} \beta_{n+4} + a_{n+7} \beta_{n+6} + \dots
\]
\end{fact}
\begin{proof} Since $\beta_{n+1} = a_{n+1} \beta_n + \beta_{n-1}$, we have
\begin{align*}
-\beta_n &= (a_{n+1}-1) \beta_{n} + \beta_{n-1} - \beta_{n+1} \\
& = \beta_{n-1}+(a_{n+1}-1) \beta_{n} + a_{n+3} \beta_{n+2} + a_{n+5} \beta_{n+4} + a_{n+7} \beta_{n+6} + \dots.
\end{align*}
\end{proof}

\noindent The following fact allows us to decide whether one real number is smaller than another if we are just given their Ostrowski representations.

\begin{fact}\label{fact:oorder}  Suppose $\beta_0 > 0$. Let $x,y \in \R$ with $x\neq y$ and let $\sum_k b_{k+1} \beta_k$  and $\sum_k c_{k+1} \beta_k$ be the Ostrowski representations of $x$ and $y$. Let $n \in \N$ be minimal such that $b_{n+1} \neq c_{n+1}$.
 Then $x < y$ iff
\begin{itemize}
\item[(i)] $b_{n+1} > c_{n+1}$ and $n$ is odd,
\item[(ii)]  $c_{n+1} > b_{n+1}$ and $n$ is even.
\end{itemize}
\end{fact}
\begin{proof}
Suppose $b_{n+1} > c_{n+1}$. Then
\[
x-y = d_{n+1}\beta_n + d_{n+2} \beta_{n+1} + \sum_{k=n+2}^{\infty} d_{k+1} \beta_k,
\]
for some $d_{n+1} >0$, $d_{n+2} < a_{n+2}$, $-a_{k} \leq d_k \leq a_{k}$ for $k\geq n+3$. It is enough show that if $n$ is odd, then $x - y\leq0$, and if $n$ is even, then $x-y\geq 0$.\\

\noindent Let $n$ be odd. Since $\beta_0>0$, we have $\beta_n < 0$ and $\beta_{n+1} > 0$ by Fact \ref{fact:beta}. Since $d_{n+1} \leq a_{n+2}-1$, we get by Fact \ref{fact:osteven} that
\begin{align*}
x - y &\leq \beta_n + (a_{n+2}-1) \beta_{n+1} + \sum_{k\geq n+2,\ k \textrm{ even}} a_{k+1} \beta_k \\
&= \beta_n - \beta_{n+1} + \sum_{k\geq n+1,\ k \textrm{ even}} a_{k+1} \beta_k \\
&= - \beta_{n+1} < 0.
\end{align*}

\noindent Let $n$ be even. Then $\beta_n > 0$ and $\beta_{n+1} < 0$. Since $d_{n+1} \leq a_{n+2}-1$, we get by Fact \ref{fact:osteven} that
\begin{align*}
x - y &\geq \beta_n  + (a_{n+2}-1) \beta_{n+1} + \sum_{k\geq n+2,\ k \textrm{ odd}} a_{k+1} \beta_k \\
&= \beta_n - \beta_{n+1} + \sum_{k\geq n+1,\ k \textrm{ odd}}^{\infty} a_{k+1} \beta_k \\
&= - \beta_{n+1} > 0.
\end{align*}
\end{proof}

\noindent A similar result holds if $\beta_0 < 0$.\\

\noindent Given two natural numbers in Ostrowski representation, it will be important for us to know how to calculate the Ostrowski representation of their sum. Assume for now that $a$ is quadratic. Since the continued fraction expansion of $a$ is periodic, there is a natural number $c:= \max_{k\in \N} a_k$. Let $\Sigma_a= \{0,\dots,c\}$ and denote by $\Sigma_a^*$ the set of words of finite length on $\Sigma_a$. Let $N \in \N$ be such that $\sum_{k=0}^{n} b_{k+1} q_{k}$ is the Ostrowski representation of $N$. Then we define $\rho_a(N)$ to be the $\Sigma_a$-word $b_{n+1}\dots b_1$. For $X \subseteq N^n$ we define \[
0^*\rho_a(X):=\{ (0^{l_1}\rho_a(N_1),\dots,0^{l_n}\rho_a(N_n)) \ : \ (N_1,\dots,N_n) \in X, l_1,\dots,l_n \in \N\}.
\]
We will now explain what it means for $X$ to be $a$-recognizable.\newline

\noindent Let $\alpha=(\alpha_1,\dots,\alpha_n) \in (\Sigma_a^*)^n$ and let $m$ be the maximal length of $\alpha_1,\dots,\alpha_n$. Then add to each $\alpha_i$ the necessary number of $0$'s to get a word $\alpha_i'$ of length $m$. Then the convolution\footnote{Here we followed the presentation in Villemaire \cite{Villemaire}. For a general definition of convolution see \cite{automata}.} of $\alpha$ is defined as the word $\alpha_1 * \dots * \alpha_n \in (\Sigma_a^n)^*$ whose $i$-th letter is the element of $\Sigma_a^n$ consisting of the $i$-th letters of  $\alpha_1', \dots, \alpha_n'$. Let $X\subseteq \N^n$. We say that $X$ is \textbf{$a$-recognizable} if $\{ \alpha_1 * \dots * \alpha_n \ : \ (\alpha_1,\dots,\alpha_n) \in 0^*\rho_a(X)\}\subseteq (\Sigma_a^n)^*$ is recognizable by a finite automaton. For a definition of a finite automaton and recognizability we refer the reader to Khoussainov and Nerode \cite{automata}.

\begin{fact}\label{fact:Zarithmetic}\cite[Theorem B]{HA} Let $a$ be quadratic. Then $\{ (x,y,z) \in \N^3 \ : \ x+y = z \}$ is $a$-recognizable.
\end{fact}

\noindent If $a$ is the golden ratio, the Ostroskwi representation is called Zeckendorf representation and in this particular case Fact \ref{fact:Zarithmetic} was first shown by Frougny in \cite{Frougny}. In \cite{EffectiveZeckendorf} Ahlbach et al. present an elementary algorithm to calculate the Zeckendorf respresentation of a sum in terms of the Zeckendorf representation of the summands. This algorithm was adjusted to give Fact \ref{fact:Zarithmetic} by Hieronymi and Terry \cite{HA}.

\section{Defining $\Cal R_{a}$ in $\Cal B$}

Let $a\in [1,2]$ be an quadratic irrational number. Since $\Cal R_{a}$ and $\Cal R_{qa}$ are interdefinable for non-zero $q\in \Q$, we can assume that $1.5 < a < 2$. Let $[a_0;a_1,\dots,a_n,\dots]$ be the continued fraction expansion of $a$. Note that since $1.5< a < 2$, $a_0 = a_1 = 1$. By Fact \ref{fact:quadraticperiodic}, the continued fraction expansion of $a$ is periodic. Hence it is of the form
\[
[a_0;a_1,\dots,a_{\xi},\overline{a_{\xi+1},\dots, a_{\xi+\nu}}],
\]
where $\nu$ is the length of the repeating block and the repeating block starts at $\xi+1$. Set $\mu := \max_i a_i$. Since $a_0=1$, we have that $\beta_0$, as is defined in Definition \ref{def:beta}, is equal to  $a - 1$ and hence positive. It also follows easily that $a =  1 + \frac{1}{\zeta_1}$. Hence the interval $[-\frac{1}{\zeta_1},1-\frac{1}{\zeta_1})$ given in Fact \ref{ostrowskireal} is equal to the interval $[1-a,2-a)$. We denote this interval by $I$. Also note that if $\sum_{k}b_{k+1} \beta_k$ is an Ostrowski representation of a real number in $I$, then $b_1 = 0$, since $a_1 = 1$. Hence $\beta_0$ is not an Ostrowski representation. The same is true for an Ostrowski representation of a natural number. Hence the Ostrowski representation of $1$ is $q_1$, and not $q_0$.
\\

The goal for this section is to show that an isomorphic copy of $\Cal R_a$ is definable in $\Cal B$. Remember that $\Cal B$ is the two sorted structure $(\N,\Cal P(\N),s_{\N},\in)$, where $s_{\N}$ is the successor function on $\N$ and $\in$ is the relation on $\N \times \Cal P(\N)$ such that $\in(t,X)$ iff $t \in X$. We recall some easy and well-known definability results for $\Cal B$. We write $\Even$ for the set of all even natural numbers and $\Odd$ for the set of all odd natural numbers. Both sets are definable in $\Cal B$. For example, $\Even$ is the unique element $X$ in $\Cal P(\N)$ such that
\[
0 \in X \wedge \forall x \in \N \ (x \in X \leftrightarrow s_{\N}(x) \notin X).
\]
Similarly, for $m,n\in \N$, the set
\[
\{ s \in \N \ : \ s=m \mod n \}
\]
is definable in $\Cal B$. Also recall that for $m,n \in \N$, we have $m < n$ iff
\[
\exists X \in \Cal P(\N) \ m \in X \wedge n \notin X \wedge \forall t \in \N \ (t \notin X \rightarrow s_{\N}(t) \notin X).
\]
Hence the order on $\N$ is definable in $\Cal B$. If $W\subseteq \Cal P(\N)$ is definable in $\Cal B$, so is the subset $W_{fin}$ of $W$ containing all finite sets in $W$. Finally, if a subset $X\subset \Cal P(\N)^n$ can be recognized by a finite automaton, then it is definable in $\Cal B$, see for example \cite[Lemma 2]{Buchi}.

\subsection*{Defining Ostrowski representations} The first step towards defining $\Cal R_a$ in $\Cal B$ will be constructing definable sets that correspond to the Ostrowski representation of both real numbers and natural numbers. This will give us two bijections between definable sets in $\Cal B$ and $I$ and $\N$.

\begin{defn} Define $A \subseteq \Cal P (\N)^{\mu}$ to be the set containing $(X_1,\dots,X_{\mu}) \in \Cal P (\N)^{\mu}$ such that
\begin{itemize}
\item $0 \notin X_i$, for $i\geq a_1-1$,
\item If $n \in X_i$, then $n \notin X_j$ for $j\neq i$,
\item $n \notin X_i$, if $0<n<\xi$ and $i > a_{n+1}$,
\item $n \notin X_i$, if $n \geq \xi$, $n+1= \xi + l \mod \nu$, $l\in\{1,\dots,\nu\}$ and $i > a_{\xi+l}$,
\item for all $m\in \N$ there exists $n \in \Even$ with $n\geq m$ such that there is $l\in\{1,\dots,\nu\}$ with
\[
n+1 = \xi + l \mod \nu \hbox{ and } n \notin X_{a_{\xi + l}}.
\]
\end{itemize}
\end{defn}

It follows from the statements about definability in $\Cal B$ we made before that $A$ is definable in $\Cal B$. Let $A_{fin} \subseteq A$ be the subset of $A$ containing all tuples $(X_1,\dots,X_{\mu})$ for which $X_i$ is finite for $i=1,\dots,\mu$. Since $A$ is definable in $\Cal B$, so is $A_{fin}$.

\begin{defn}\label{def:bk} Let $b_{k+1} : A_{fin} \to \N$ map $(X_1,\dots,X_{\mu})$ to
\[
\left\{
  \begin{array}{ll}
    i, & \hbox{if $k \in X_i$;} \\
    0, & \hbox{otherwise.}
  \end{array}
\right.
\]
If $X\in A_{fin}$, define $Z(X)$ to be the natural number
\[
Z(X) = \sum_{k \in \N} b_{k+1}(X) q_{k}.
\]
\end{defn}

\noindent Note that by uniqueness of Ostrowski representations (see Fact \ref{ostrowski}), the map $Z : A_{fin} \to \N$ is bijective. Hence $Z$ has an inverse which we denote by $Z^{-1}$. Also note that the relations
$b_{k+1}(X) < b_{k+1}(Y)$ and $b_{k+1}(X) = b_{k+1}(Y)$ on $\N \times A \times A$ are definable in $\Cal B$.

\begin{defn} Let $X \in A$. Define $O(X)$ to be the real number in $I$
\[
O(X) = \sum_{k \in \N} b_{k+1}(X) \beta_{k}.
\]
\end{defn}

\noindent By the uniqueness of the Ostrowski representations (see Fact \ref{ostrowskireal}), the map $O : A \to I$  is bijective. Hence $O$ has an inverse which we denote by $O^{-1}$.

\begin{lem}\label{lemma:fzo} Let $X \in A_{fin}$. Then $aZ(X)- O(X) \in \N$.
\end{lem}
\begin{proof} Let $X \in A_{fin}$. Then
\begin{align*}
aZ(X) - O(X) &=  \sum_{k=0}^n b_{k+1}(X) aq_{k} - \sum_{k=0}^n b_{k+1}(X) \beta_{k}\\
&=  \sum_{k=0}^n b_{k+1}(X) q_{k}a - \sum_{k=1}^n b_{k+1}(X) (q_ka - p_{k}) = \sum_{k=1}^n b_{k+1}(X) p_{k} \in \N.
\end{align*}
\end{proof}

\noindent It is now a good point to outline the strategy for defining $\Cal R_a$ in $\Cal B$. We have already constructed a bijection $O$ between an interval $I$ and
the definable set $A$ in $\Cal B$. Moreover the map $aZ: A_{fin} \to \Z a$ that maps $X \in A_{fin}$ to $aZ(X)$ is a bijection. In the following we will amalgamate these two bijection to a single bijection between $\R$
 and a set $C$ definable in $\Cal B$. The reason we choose the map $aZ$ and not the map $Z$ to start with, is Lemma \ref{lemma:fzo}. Vaguely speaking, because $aZ(X)-O(X)\in \N$, we will be able to recover $\N$ from the images $O(A)$ and $aZ(A_{fin})$.

\subsection*{Defining order and addition} After defining $A$ and $A_{fin}$, we will now discuss how to define order and addition such that the maps $O$ and $Z$ respect order and addition on $I$ and $\N$.

\begin{defn} Let $\oplus : A_{fin} \times A_{fin} \to A_{fin}$ be given by
\[
X \oplus Y := Z^{-1}\big(Z(X) + Z(Y)\big).
\]
\end{defn}

\begin{lem} The function $\oplus$ is definable in $\Cal B$.
\end{lem}
\begin{proof} It follows immediately from Fact \ref{fact:Zarithmetic} that the graph of $\oplus$ can be recognized by a finite automaton. Hence it is definable in $\Cal B$.
\end{proof}

\begin{defn} Let $X,Y \in A_{fin}$ and $X\neq Y$. Let $k\in \N$ be the maximal natural number such that $b_{k+1}(X)\neq b_{k+1}(Y)$. We say $X \prec_Z Y$ if $b_{k+1}(X) < b_{k+1}(Y)$.
\end{defn}

\noindent It follows immediately from the comment after Definition \ref{def:bk} that $\prec_Z$ is definable in $\Cal B$. The following Lemma follows immediately from Fact \ref{fact:zorder}.

\begin{lem}\label{lemma:zorder} Let $X,Y \in A_{fin}$. Then $X \prec_Z Y$ iff $Z(X) < Z(Y)$.
\end{lem}

\noindent Hence $Z$ is an isomorphism between $(A_{fin},\prec_Z,\oplus)$ and $(\N,<,+)$.

\begin{lem}\label{lemma:oplus} Let $X,Y \in A_{fin}$. Then
\[
O(X \oplus Y) = O(X) + O(Y) \mod 1.
\]
\end{lem}
\begin{proof} By Lemma \ref{lemma:fzo}, there is $N \in \N$ such that
\begin{align*}
O(X) + O(Y) - O(X \oplus Y) = a (Z(X)+Z(Y) - Z(X\oplus Y)) - N.
\end{align*}
By definition of $\oplus$, the right hand side of the previous equation is equal to $-N$.
\end{proof}



\begin{defn} Let $X,Y \in A$ be such that $X \neq Y$. Let $k\in \N$ be the minimal natural number such that $b_{k+1}(X)\neq b_{k+1}(Y)$. We say $X \prec_O Y$ if one of the following conditions hold:
\begin{itemize}
\item[(i)] $b_{k+1}(X) > b_{k+1}(Y)$ and $k$ is odd,
\item[(ii)]  $b_{k+1}(Y) > b_{k+1}(X)$ and $k$ is even.
\end{itemize}
\end{defn}

\noindent It is easy to see that $\{ (X,Y) \in A \ : \ X \prec_O Y \}$ is definable in $\Cal B$. The following Lemma follows immediately from Fact \ref{fact:oorder}.

\begin{lem}\label{lemma:oorder} Let $X,Y \in A$. Then $X \prec_O Y$ iff $O(X) < O(Y)$.
\end{lem}

\begin{cor}\label{lemma:density} Let $X,Y \in A$ be such that $X \prec_O Y$. Then there is $Z \in A_{fin}$ such that $X \prec_O Z \prec_O Y$.
\end{cor}
\begin{proof} By Lemma \ref{lemma:oorder}, $O(X) < O(Y)$. Take $Z_0 \in A$ such that $O(X) < O(Z_0) < O(Y)$.
Let $\varepsilon \in \R_{>0}$ be such that $\varepsilon < \min \{ O(Z_0) - O(X), O(Y) - O(Z_0)\}$. Take $n\in \N$ such that
\[
\max \{ \sum_{k>n, k \textrm{ even}} a_{k+1}\beta_k, -\sum_{k\geq n, k \textrm{ odd}} a_{k+1}\beta_k \} < \varepsilon.
\]
Let $Z_{0,1},\dots,Z_{0,\mu} \in \Cal P(\N)$ such that $Z_0 = (Z_{0,1},\dots,Z_{0,\mu})$. Set
\[
Z := ((Z_{0,1}\cap [0,n],\dots,Z_{0,\mu}\cap [0,n]).
\]
Then
\begin{align*}
O(Z) &= \sum_{k = 0}^{n} b_{k+1}(Z) \beta_k = \sum_{k = 0}^{\infty} b_{k+1}(Z_0) \beta_k - \sum_{k=n+1}^{\infty} b_{k+1}(Z)\beta_k\\
& = O(Z_0) - \sum_{k=n+1}^{\infty} b_{k+1}(Z_0)\beta_k.
\end{align*}
Since $|\sum_{k=n+1}^{\infty} b_{k+1}(Z_0)\beta_k |< \varepsilon$, we have $O(X) < O(Z) < O(Y)$. Again by Lemma \ref{lemma:oorder}, $X \prec_O Z \prec_O Y$.
\end{proof}

\noindent We now use this density to extend $\oplus$ to $A$.

\begin{defn} Define $+_1 : I \times I \to I$ be the function that maps $(c_1,c_2) \in I^2$ to the unique element $d \in I$ such that $d = c_1 + c_2 \mod 1$.
Let $\oplus : A \times A \to A$ be function that maps $(X,Y)$ to $O^{-1}(O(X) +_1 O(Y))$.
\end{defn}

The following Lemma follows immediately from the definition of $\oplus$.

\begin{lem}\label{cor:oplus} Let $X,Y \in A$. Then
\[
O(X\oplus Y) = O(X) + O(Y) \mod 1.
\]
\end{lem}

\begin{lem}\label{lem:defoplus123} The map $\oplus : A \times A \to A$ is definable in $\Cal B$.
\end{lem}
\begin{proof} Consider the following two structures,
\[
\Cal M_1 := (A,\prec_O,\oplus|_{A_{fin}},A_{fin}), \Cal M_2 := (I,<,+_1|_{O(A_{fin})},O(A_{fin})).
\]
By Lemma \ref{lemma:oplus} and Lemma \ref{lemma:oorder}, the map $O : \Cal M_1 \to \Cal M_2$ is an isomorphism. Let $\Cal T$ be the topology on $I$ whose basic open sets are
the intervals $(c_1,c_2)$, if $c_1,c_2 \in I$ and $c_1<c_2$, and the sets $[1-a,c_2) \cup (c_1,2-a)$, if $c_1,c_2 \in I$ and $c_1 > c_2$. It is immediate that the topological closure with respect to $\Cal T$ of a set definable in $\Cal M_2$ is definable in $\Cal M_2$ as well. By Lemma \ref{lemma:oorder} and Lemma \ref{lemma:density}, $O(A_{fin})$ is dense in $I$ with respect to $\Cal T$. Because of the continuity of $+_1$ with respect to $\Cal T$, the topological closure in $\Cal T$ of the graph $+_1|_{O(A_{fin})}$ is the graph of $+_1$. Hence $+_1$ is definable in
$\Cal M_2$. Since $O$ is an isomorphism, $\oplus$ is definable in $\Cal M_1$ and hence in $\Cal B$.
\end{proof}


\noindent Thus $(A,\oplus)$ forms a group. The neutral element $\mathbf{0}$ is $(\emptyset,\dots,\emptyset)$. We write
\[
\mathbf{1} := (\{1\},\emptyset,\dots,\emptyset).
\]
Note that $O(\mathbf{1})=\beta_1=a-2$. For $X\in A$, we denote the inverse of $X$ with respect to $\oplus$ by $\ominus X$, that means $X \oplus (\ominus X) = \mathbf{0}$. As usual, for $X,Y \in A$ we will write $X\ominus Y$ for $
X\oplus (\ominus Y)$.

\subsection*{Modifying $O$} We have constructed an isomorphism $O$ between $(A,\prec_O,\oplus)$ and $(I,<,+_1)$. In the following this isomorphism will be modified to an isomorphism $S$ whose range is $([0,1),<,+ \mod 1)$ instead of $(I,<,+_1)$. Here $+ \mod 1 : [0,1)^2 \to [0,1)$ is the map that takes $(x,y) \in [0,1)^2$ to the unique $z \in [0,1)$ such that $x+y = z \mod 1$.

\begin{defn} Let $X,Y \in A$, we write $X \prec_1 Y$ if
\begin{itemize}
\item $Y \prec_O \mathbf{1} \preceq_O X$,
\item $X,Y \prec_O \mathbf{1}$ and $X \prec_O Y$,
\item $X,Y \succ_O \mathbf{1}$ and $X \prec_O Y$
\end{itemize}
Let $\oplus_1 : A \times A \to A$ be the map that takes $(X,Y) \in A^2$ to $(X\oplus Y) \ominus \mathbf{1}$.
Let $S : A \to [0,1)$ maps $X \in A$ to
\[
\left\{
  \begin{array}{ll}
    O(X) - O(\mathbf{1}), & \hbox{ if $\mathbf{1} \preceq_O X$;} \\
    O(X) - O(\mathbf{1}) + 1, & \hbox{otherwise.}
  \end{array}
\right.
\]

\end{defn}

\begin{lem}\label{lemma:siso} The map $S: (A,\prec_1,\oplus_1) \to ([0,1),<,+ \mod 1)$ is an isomorphism.
\end{lem}
\begin{proof} Since $O(\mathbf{1})=a-2$ and $I=[1-a,2-a)$, we directly get that
\[
S(\{ X \in A \ : \ \mathbf{1} \preceq_O X \}) = [0,4-2a) \hbox{ and } S(\{ X \in A \ : \ X \prec_O  1 \}) = [4-2a,1).
\]
We have $4-2a < 1$, since $1.5 < a < 2$. It follows immediately that $S(X) < S(Y)$ iff $X \prec_1 Y$. Since $O$ is bijective, it is easy to see that $S$ is bijective. Note that $S(X)$ is the unique $c\in [0,1)$ with $c =O(X) - a \mod 1$. Let $X,Y \in A$. Then
\begin{align*}
S(X) + S(Y) &= O(X) + O(Y) - 2O(\mathbf{1}) \mod 1 \\
&= O((X\oplus Y)\ominus \mathbf{1}) -  O(\mathbf{1}) \mod 1\\
&= S(X \oplus_1 Y) \mod 1.
\end{align*}

\end{proof}

Hence $(A,\oplus_1)$ is a group and its neutral element is $\mathbf{1}$. As above, for $X \in A$ we will write $\ominus_1 X$ for the inverse element of $X$ in $A$ with respect to $\oplus_1$. Thus $\ominus_1 X$ is the unique element in $A$ such that $(\ominus_1 X) \oplus_1 X = \mathbf{1}$.

\begin{lem}\label{lemma:splus} Let $X,Y \in A$. Then $\ominus_1 X \preceq_1 Y$ iff  $S(X)+ S(Y) \geq 1$.
\end{lem}
\begin{proof} Since $1- S(X) = -S(X) \mod 1$ and $S$ is group homomorphism, we have $1-S(X) = S(\ominus_1 X)$.
Hence we have by Lemma \ref{lemma:siso}
\begin{align*}
S(X) + S(Y) \geq 1 &\hbox{ iff } S(Y)\geq 1 - S(X)\\
&\hbox{ iff } S(Y)\geq S(\ominus_1 X) \\
&\hbox{ iff } Y\succeq_1 \ominus_1 X.
\end{align*}
\end{proof}

\begin{cor}\label{cor:splus} Let $X,Y\in A$. Then
\[
S(X\oplus_1 Y) = \left\{
  \begin{array}{ll}
    S(X) + S(Y) , & \hbox{if $\ominus_1 X \preceq_1 Y$;} \\
    S(X) + S(Y) -1, & \hbox{otherwise.}
  \end{array}
\right.
\]
\end{cor}

\subsection*{Recovering $\N$} We have now established that we can define order and addition on $A$ and $A_{fin}$ such that $O$, $S$ and $Za$ become isomorphisms. Vaguely speaking, the next step is to recover $\N$ from $O(A)$ and $Za(A_{fin})$. We will find a set $B$ definable in $\Cal B$, a definable order $\prec_B$ on $B$, a definable operation $\oplus_B :B \times B \to B$ and a map $R: B \to \N$ such that $R$ is an isomorphism between
$(B,\prec_B,\oplus_B)$ and $(\N,<,+)$. It will be crucial later that the isomorphism $R$ arises naturally from $O$ and $Za$.

\begin{lem}\label{lemma:nat} Let $X\in A_{fin}$. Then $\N \cap \big(aZ(X) ,(Z(X) +1)a \big)$ is
\begin{align*}
 \left\{
                                                \begin{array}{ll}
                                                  \{aZ(X) - O(X)+1, aZ(X) - O(X)+2\} , & \hbox{if $X\prec_O \mathbf{1}$} \\
                \{aZ(X) - O(X)+1\}, & \hbox{if $\mathbf{1}\preceq_O X \preceq_O \mathbf{0}$;} \\
                                                  \{aZ(X) - O(X), aZ(X) - O(X)+1\}, & \hbox{otherwise.}
                                                \end{array}
                                              \right.
\end{align*}
\end{lem}
\begin{proof} Let $X \in A_{fin}$. Since $1.5<a<2$, there are at most two natural numbers between $aZ(X)$ and $a(Z(X)+1)$. Moreover, since $1.5<a<2$, $3-a<a$. We will use this observation repeatedly throughout this proof.
By Lemma \ref{lemma:fzo}, $aZ(X) -O(X) \in \N$ and so are $aZ(X) - O(X)+1$ and $aZ(X) - O(X)+2$. Since $1-a\leq O(X)< 2-a$, we also have that
\[
aZ(X) < aZ(X) + O(X) + 1 < a(Z(X)+1).
\]
We just have to determine which of the other two natural numbers fall into the interval we are considering. First consider the case that $X \prec_O \mathbf{1}$. Since $O(\mathbf{1})= a - 2$, we have $O(X) < a -2$ by Lemma \ref{lemma:oorder}. Since $1-a\leq O(X) < a-2$, we get that
\[
aZ(X) + O(X) < aZ(X) < aZ(X) + O(X) +2 < a(Z(X)+1).
\]

\noindent Now suppose that $\mathbf{1}\preceq_O X \preceq_O \mathbf{0}$. Since $O(\mathbf{0})=0$ and $O(\mathbf{1})= a - 2$, we have $a - 2 \leq O(X)\leq 0$ by Lemma \ref{lemma:oorder}. Hence
\[
aZ(X) + O(X) \leq aZ(X) \hbox{ and } a(Z(X)+1) \leq aZ(X) + O(X) +2.
\]

\noindent Finally consider the case that $\mathbf{0} \prec_O X$. Since $O(\mathbf{0})=0$, we have $0 <O(X)$ by Lemma \ref{lemma:oorder}. Since $O(X)<2-a$, we have
\[
 aZ(X) < aZ(X) + O(X)  < aZ(X) + O(X) +1 < a(Z(X)+1).
\]
\end{proof}

\begin{defn} Let $B \subseteq A_{fin} \times \{0,1,2\}$ be defined as the set of all pairs $(X,i)$ that have one of the following properties:
\begin{itemize}
\item [(i)] $X=\mathbf{0}$ and $i=0$,
\item [(ii)] $X\prec_O \mathbf{1}$ and $i \in \{1,2\}$,
\item [(iii)]$\mathbf{1}\preceq_O X \preceq_O \mathbf{0}$ and $i=1$,
\item [(iv)] $\mathbf{0} \prec_O X$  and $i \in \{0,1\}$.
\end{itemize}
Let $R: B \to \N$ map $(X,i)$ to $aZ(X) - O(X) + i$.
\end{defn}

\begin{lem}\label{lemma:Rbij} $R$ is a bijection.
\end{lem}
\begin{proof} By Lemma \ref{lemma:nat}, $\N a \cap \N = \{0\}$ and the fact that $Z: A_{fin} \to \N$ is a bijection, $R$ maps $B\setminus \{(\mathbf{0},0)\}$ bijectively to $\N_{>0}$. Hence $R$ is bijective, since $R((\mathbf{0},0))=0$.
\end{proof}

\begin{defn} Let $(X,i),(Y,j) \in B$. We write $(X,i) \prec_B (Y,j)$ if either
\begin{itemize}
\item $X=Y$ and $i<j$ or
\item $X\prec_Z Y$.
\end{itemize}
 Let $s_B: B \to B$ map $Z \in B$ to its $\prec_B$-successor in $B$. Let $p_B: B \setminus \{ (\mathbf{0},0)\} \to B$ map $Z \in B$ to its $\prec_B$-predecessor in $B$.
\end{defn}

\noindent Since $\prec_Z$ well-orders $A_{fin}$, $\prec_B$ well-orders $B$. Hence the successor and predecessor function are well-defined. Moreover, by Lemma \ref{lemma:zorder}, we have that for $Z_1,Z_2 \in B$, $Z_1 \prec_B Z_2$ iff $R(Z_1) < R(Z_2)$. Since $R$ is a bijection, we have for $Z \in B$
\begin{equation}\label{eq:rplus1}
R(s_B(Z))=R(Z) + 1.
\end{equation}
We will use the following notation: we write $s_B^0$ for the identity on $B$, and for $i \in \N_{>0}$, we write $s^{-i}_B$ for $i$-th iterate of $p_B$ and $s^{i}_B$ for the $i$-th iterate of $s_B$.\\


\noindent We will now define $\oplus_B : B \times B\to B$ such that $R$ is an isomorphism from $(B,\prec_B,\oplus)$ to $(\N,<,+)$. For convenience set $\mathbf{e}:=\ominus \mathbf{1}$. Since $O(\mathbf{1})=a-2$, we have that $O(\mathbf{e})=1-a$ by Lemma \ref{cor:oplus}. Hence $\mathbf{e}$ is the left endpoint of $I$ and is $\prec_O$-minimal in $A$ by Lemma \ref{lemma:oorder}.

\begin{defn} For $X,Y \in A$, we define $r(X,Y) \in \{0,1,2\}$ to be
\[
\left\{
  \begin{array}{ll}
    0, & \hbox{if $X \prec_O \mathbf{e} \ominus Y$ and $Y\preceq_O \mathbf{0}$;} \\
    2 & \hbox{if $\mathbf{e} \ominus Y \preceq_O X$ and $Y \succ_O \mathbf{0}$;} \\
    1, & \hbox{otherwise.}
  \end{array}
\right.
\]
\end{defn}

\begin{defn} Let $(X,i),(Y,j) \in B$. We define $\oplus_B : B \times B \to B$ by
\[
(X,i)\oplus_B (Y,j) := s^{i+j+r(X,Y)-2}((X\oplus Y,1)).
\]
\end{defn}

\begin{lem}\label{lemma:lminus} Let $X \in A$. Then
\[
O(\mathbf{e} \ominus X) = \left\{
                            \begin{array}{ll}
                              O(\mathbf{e})-O(X), & \hbox{ if $X \preceq_O \mathbf{0}$;} \\
                              O(\mathbf{e})-O(X)+1, & \hbox{otherwise.}
                            \end{array}
                          \right.
\]
\end{lem}
\begin{proof} Suppose $X \preceq_O \mathbf{0}$. Since $1-a \leq O(X) \leq 0$, we have
\[
1- a \leq O(\mathbf{e}) - O(X) = 1 - a -  O(X) \leq 0.
\]
Hence $O(\mathbf{e}) - O(X) \in I$ and thus $O(\mathbf{e} \ominus X) = O(\mathbf{e}) - O(X)$ by Lemma \ref{cor:oplus}.\newline
Suppose $X \succ_O \mathbf{0}$. Then $0 < O(X) < 2-a$, and thus $0<a-1 < 1-O(X) < 1$. Hence
\[
O(\mathbf{e}) < O(\mathbf{e}) - O(X) + 1 < 1-a+ 1 = 2-a.
\]
Hence $O(\mathbf{e}) - O(X) + 1 \in I$ and therefore $O(\mathbf{e} \ominus X)=O(\mathbf{e}) - O(X) + 1 $.
\end{proof}

\begin{lem}\label{lemma:or} Let $X,Y \in A$. Then $O(X) + O(Y) = O(X\oplus Y)+r(X,Y)-1$.
\end{lem}
\begin{proof} We first consider the case that $Y \preceq_O \mathbf{0}$. Suppose that $X \prec_O \mathbf{e} \ominus Y$. Hence $r(X,Y)=0$. By Lemma \ref{lemma:lminus}
\[
O(X) + O(Y) < O(\mathbf{e} \ominus Y) + O(Y) = O(\mathbf{e}) \leq O(X\oplus Y).
\]
Hence $O(X)+O(Y)+1 \in I$ and thus $O(X\oplus Y) - 1 = O(X) + O(Y)$. Suppose now that $X \succeq_O \mathbf{e} \ominus Y$. Hence $r(X,Y)=1$. By Lemma \ref{lemma:lminus}
\[
O(X) \geq O(X) + O(Y) \geq O(\mathbf{e} \ominus Y) + O(Y) = O(\mathbf{e}).
\]
Hence $O(X) + O(Y) \in I$ and thus $O(X\oplus Y) = O(X) + O(Y)$. \\
Now consider the case that $Y \succ_O \mathbf{0}$. Suppose that $\mathbf{e} \ominus Y\preceq_O X$. Hence $r(X,Y)=2$. Again by Lemma \ref{lemma:lminus}
\[
O(X) + O(Y) \geq O(Y) +  O(\mathbf{e}) - O(Y) +1  = O(\mathbf{e})+1 = 2 -a > O(X\oplus Y).
\]
Thus $O(X)+O(Y)= O(X\oplus Y)+1$. Finally suppose that $\mathbf{e} \ominus Y \succ_O X$. Then we have $r(X,Y)=1$. By Lemma \ref{lemma:lminus}
\[
O(X) \leq O(X) + O(Y) < O(Y) +  O(\mathbf{e})+1 - O(Y)  = O(\mathbf{e}) + 1=2-a.
\]
Hence $O(X)+O(Y) \in I$ and hence $O(X\oplus Y)=O(X) + O(Y)$.
\end{proof}

\begin{lem} Let $Z_1,Z_2 \in B$. Then $R(Z_1\oplus_B Z_2) = R(Z_1) + R(Z_2)$.
\end{lem}
\begin{proof} Let $(X,i),(Y,j) \in B$. Then by \eqref{eq:rplus1} and Lemma \ref{lemma:or}
\begin{align*}
R((X,i)\oplus_B (Y,j)) &= R(s^{i+j+r(X,Y)-2}(X\oplus Y,1))\\
&= R((X\oplus Y),1) + i + j + r(X,Y)-2\\
&=Z(X\oplus Y)a + O(X\oplus Y) + i + j + r(X,Y)-1\\
&=aZ(X) + Z(Y)a + O(X) + O(Y) + i + j\\
&=R((X,i)) + R((Y,j)).
\end{align*}
\end{proof}

\begin{cor}\label{cor:riso} The map $R: (B,\prec_B,\oplus_B) \to (\N,<,+)$ is an isomorphism.
\end{cor}

\begin{table}[b]
\begin{tabular}{c|cc}

  Isomorphism & Domain & Codomain \\
    \hline
  $O$ & $A$ & $I$ \\
  $Z$ & $A_{fin}$ & $\N$ \\
  $S$ & $A$ & $[0,1)$ \\
  $R$ & $B$ & $\N$ \\
  $T$ & $C$ & $R_{\geq 0}$ \\
\end{tabular}
  \caption{A list of the isomorphisms and their domain and codomain}
\end{table}
\subsection*{Amalgamating $R$ and $S$} We have constructed two isomorphisms $R : B \to \N$ and $S: A \to [0,1)$. We define $T : B \times A \to \R_{\geq 0}$ as the map that takes $(Z,X)\in B\times A$ to $R(Z) + S(X)$. It follows immediately from Lemma \ref{lemma:siso} and Lemma \ref{lemma:Rbij} that $T$ is bijective. We will now construct two definable subsets $A',B'$ of $B\times A$, a definable relation $\prec_C$ and a definable operation $\oplus_C : (B\times A)^2\to B\times A$ such that $T$ is an isomorphism between $(B\times A,\prec_C,\oplus_C,B',A')$ and $(\R_{\geq 0},<,+,\N,\N a)$.

\begin{defn} Set $C:=B\times A$. Let $A'\subseteq C$ be
\begin{align*}
\{ (p_B^2(X,1), X \oplus \mathbf{1}) \ & : X \in A_{fin}, X\prec_O \mathbf{0} \} \\
&\cup \{ (p_B(X,1),X \oplus \mathbf{1}) \ : X \in A_{fin}, X\succeq_O \mathbf{0} \},
\end{align*}
and let $B'\subseteq C$ be the set $\{ (Z,\mathbf{1}) : \ Z \in B \}$.\newline
\end{defn}

\begin{lem}  The map $T : (C,B',A') \to (\R_{\geq 0},\N,\N a)$ is an isomorphism.
\end{lem}
\begin{proof} We first show that $T(B') =\N$. Let $(Z,\mathbf{1}) \in B'$. Then
\[
T(Z,\mathbf{1}) = R(Z)+S(\mathbf{1}) = R(Z) + 0 = R(Z) \in \N.
\]
Since $R : B \to \N$ is bijective by Lemma \ref{lemma:Rbij}, we have $T(B') = \N$ .\\

\noindent We now establish that $T(A')= \N a$. Let $X \in A_{fin}$. Then
\begin{equation}\label{eq:1}
O(X \oplus \mathbf{1}) = O(X) + O(\mathbf{1}) \mod 1.
\end{equation}
Suppose that $X \succeq _O \mathbf{0}$. Since $0 \leq O(X) < 1$ and $S(X\oplus \mathbf{1})\in [0,1)$, we get $S(X\oplus \mathbf{1}) = O(X)$ by \eqref{eq:1}. Hence
\begin{align*}
T(p_B(X,1), X \oplus \mathbf{1}) &= R(p_B(X,1))  + S(X \oplus \mathbf{1})\\
&= aZ(X) - O(X) + O(X) = aZ(X).
\end{align*}
Now suppose $X \prec_O \mathbf{0}$. Since $0< O(X)+1 < 1$, we have $S(X\oplus \mathbf{1})  = O(X)+1$ by \eqref{eq:1}. Hence
\begin{align*}
T(p_B^2(X,1), X \oplus \mathbf{1}) &= R(p_B^2(X,1))  + S(X\oplus \mathbf{1}) \\
&= aZ(X) - O(X) -1 + O(X)+1 = aZ(X).
\end{align*}
Since $Z : A_{fin} \to \N$ is bijective, we have $T(A')= a\N$.\\
\end{proof}

\begin{defn} Let $(Z_1,X_1),(Z_2,X_2) \in C$, we define
\[
(Z_1,X_1) \oplus_C (Z_2,X_2) := \left\{
                                  \begin{array}{ll}
                                    (s_B(Z_1\oplus_B Z_2),X_1 \oplus_1 X_2), & \hbox{if $\ominus_1 X_1 \preceq_1 X_2$;} \\
                                    (Z_1\oplus_B Z_2,X_1 \oplus_1 X_2), & \hbox{ otherwise.}
                                  \end{array}
                                \right.
\]
We say $(Z_1,X_1) \prec_C (Z_2,X_2)$ if $Z_1 \prec_B Z_2$ or ($Z_1=Z_2$ and $X_1 \prec_1 X_2$).\newline
\end{defn}

\begin{lem} The map $T : (C,\prec_C,\oplus_C,B',A') \to (\R_{\geq 0},<,+,\N,\N a)$ is an isomorphism.
\end{lem}
\begin{proof} Let $(Z_1,X_1), (Z_2,X_2) \in C$. By the definition of the maps $R$ and $S$, we have $R(B) = \N$ and $S(A)=[0,1)$. Thus we see directly that $T(Z_1,X_1)< T(Z_2,X_2)$ holds iff either $R(Z_1)< R(Z_2)$ holds or, $R(Z_1)=R(Z_2)$ and $S(Z_1) < S(Z_2)$ hold. By Corollary \ref{cor:riso} and Lemma \ref{lemma:siso} we have $T(Z_1,X_1)< T(Z_2,X_2)$ iff  $(Z_1,X_1) \prec_C (Z_2,X_2)$.\\

\noindent It is left to show that $T((Z_1,X_1) \oplus_C (Z_2,X_2)) = T(Z_1,X_1) +  T(Z_2,X_2)$. First suppose that $\ominus_1 X_1 \preceq_1 X_2$. Then by Corollary \ref{cor:splus}
\begin{align*}
T((Z_1,X_1) \oplus_C (Z_2,X_2)) &= T(s_B(Z_1\oplus_B Z_2), X_1 \oplus_1 X_2)\\
&= R(s_B(Z_1\oplus_B Z_2)) + S(X_1 \oplus_1 X_2)\\
&= R(Z_1\oplus_B Z_2) + 1 + S(X_1) + S(X_2) - 1\\
&= R(Z_1) + R(Z_2) + S(X_1) + S(X_2)\\
&=T(Z_1,X_1) +  T(Z_2,X_2).
\end{align*}
If $\ominus_1 X_1 \succ_1 X_2$, we have by Corollary \ref{cor:splus}
\begin{align*}
T((Z_1,X_1) \oplus_C (Z_2,X_2)) &= T(Z_1\oplus_B Z_2, X_1 \oplus_1 X_2)\\
&= R(Z_1\oplus_B Z_2) + S(X_1 \oplus_1 X_2)\\
&= R(Z_1\oplus_B Z_2) + S(X_1) + S(X_2)\\
&=T(Z_1,X_1) +  T(Z_2,X_2).
\end{align*}

\end{proof}

It follows easily from the previous Lemma that an isomorphic copy of $\Cal R_{a}$ is definable in $\Cal B$. Hence Theorem C holds.

%
%



\section{Defining $\Cal B$ in $\Cal R_{a}$}

Let $a \in \R\setminus \Q$. Since $\Cal R_{a}$ and $\Cal R_{qa}$ are interdefinable for non-zero $q\in \Q$, we can assume that $1.5 < a < 2$. In this section, we will show that an isomorphic copy of $\Cal B$ is definable in $\Cal R_{a}$. We do not require $a$ to be quadratic.\\

\noindent Since $1< a < 2$, we have $a = 1 + \frac{1}{\zeta_1}$ and hence $[1-a,2-a)=[-\frac{1}{\zeta_1},1-\frac{1}{\zeta_1})$. Recall that we denote this interval $I$. It is obviously definable in $\Cal R_a$. Moreover, since $1<a<2$, $\beta_0=a-1 > 0$.

\begin{defn} Let $U$ be the set of all pairs $(p,qa) \in \N \times \N a$ with
\[
\forall p' \in \N \ \forall q'a \in \N a (a \leq q'a \leq qa \wedge (p, qa) \neq (p',q'a)) \rightarrow |q'a - p'| > |qa - p|
\]
\end{defn}
\noindent Note that $U$ is definable in $\Cal R_a$. By Fact \ref{bestrational} the set $\{ q_ka \ : k > 0\}$ is the projection on the second coordinate of $U$ and hence definable in $\Cal R_{a}$. We denote this set by $V$. Since $V$ is definable, the successor function $s_V$ on $V$ is definable as well. Note for every $q_la\in V$ we have $s_V(q_la) = q_{l+1}a$.

\begin{defn} Let $f: \N a \to \R$ map $na$ to $na - m$, where $m$ is the unique natural number such that $na - m \in I$.
\end{defn}

\noindent Obviously, $f$ is well-defined and definable in $\Cal R_{a}$.

\begin{lem}\label{lemma:zeckenf} Let $na \in \N a$ and let $\sum_{k} b_{k+1} q_k$ be the Ostrowski representation of $n$. Then
\[
f(na) = \sum_{k} b_{k+1} \beta_k.
\]
\end{lem}
\begin{proof} Let $m := \sum_{k} b_{k+1} p_{k}$. Then
\[
n a - m = \sum_{k} b_{k+1} (q_k a - p_{k}) = \sum_{k} b_{k+1} \beta_k \in I.
\]
\end{proof}

\noindent So in particular, $f(q_ka) = \beta_k$ for every $k\in \N$.

\begin{cor}\label{cor:odddef} The set $\{ q_ka \ : \ k \hbox{ odd} \}$ is definable in $\Cal R_{a}$.
\end{cor}
\begin{proof} Since $\beta_k < 0$ iff $k$ is odd, we have by Lemma \ref{lemma:zeckenf} that $f(q_ka) < 0$ iff $k$ is odd. Hence the above set is equal to $\{ na \in V : \ f(na) < 0\}$.
\end{proof}

\begin{defn} Define $g : V \to \R^2$ by
\[
q_{l} a \mapsto \left\{
            \begin{array}{ll}
              (-(\beta_{l}+\beta_{l+1}),-\beta_{l+1}), & \hbox{if $l$ is even,}\\
              (-\beta_{l},-(\beta_{l} + \beta_{l+1})), & \hbox{otherwise.}\\
            \end{array}
          \right.
\]
\end{defn}
\noindent By Lemma \ref{lemma:zeckenf}, Corollary \ref{cor:odddef} and $s_V(q_la)=q_{l+1}a$, the function $g$ is definable. For ease of notation, we will write $g_l=(g_{l,1},g_{l,2})$ for $g(q_la)$.

\begin{lem}\label{lemma:smallo} Let $n \in \N$ and $c\in \R$ be such that $\sum_{k=0}^{\infty} b_{k+1}\beta_k$ is the Ostrowski representation of $c$.
If $-\beta_n < c < -(\beta_n + \beta_{n+1})$ or $-(\beta_n + \beta_{n+1}) < c< -\beta_{n}$, then $b_{k+1} = 0$ for all $k\leq n$.
\end{lem}
\begin{proof} Suppose $-\beta_{n} < c< -(\beta_n+\beta_{n+1})$. Since $\beta_{n+1} < 0$, $n$ is even. By Fact \ref{fact:osteven} the Ostrowski representation of $-\beta_n$ is
\[
 a_{n+2} \beta_{n+1} + a_{n+4} \beta_{n+3} + a_{n+6} \beta_{n+5} + \dots.
\]
Since $-\beta_n < c$, we have $b_{k+1} = 0$ for all odd $k\leq n$ by Fact \ref{fact:oorder}. By Fact \ref{fact:osteven} the Ostrowski representation of $-(\beta_n + \beta_{n+1})$ is
\[
(a_{n+2}-1) \beta_{n+1} + a_{n+3} \beta_{n+2} + a_{n+5} \beta_{n+4} + a_{n+7} \beta_{n+6} + \dots.
\]
Since $c< -(\beta_n + \beta_{n+1})$, we get $b_{k+1}=0$ for all even $k \leq n$ by Fact \ref{fact:oorder}. Hence $b_{k+1} = 0$ for all $k\leq n$. The case that $-(\beta_n + \beta_{n+1}) < c < -\beta_{n}$ can be handled similarly.
\end{proof}

\begin{lem}\label{lemma:ostzeck} Let $l,n \in \N$ and $c \in I$ such that $n < q_{l+1}$ and
\[
f(na) + g_{l,1}  \leq c < f(na)  + g_{l,2}.
\]
and let $\sum_{k=0}^{\infty} b_{k+1} \beta_k$ be the Ostrowski representation of $c$. Then $\sum_{k=0}^{l} b_{k+1} q_k$ is the Ostrowski representation of $n$.
\end{lem}
\begin{proof} Let $l$ be even and $n \geq q_l$. Then by definition of the function $g$, we have
\[
-(\beta_l + \beta_{l+1})<c -f(na)< -\beta_{l+1}.
\]
Hence by Lemma \ref{lemma:smallo} the Ostrowski representation of $c-f(na)$ is
$
\sum_{k=l+1}^{\infty} c_{k+1} \beta_k,
$
for some $c_{k+1} \in \{0,...,a_{k+1}\}$. Now let $\sum_{k=0}^{l} c_{k+1} q_k$ be the Ostrowski representation of $n$. By Lemma \ref{lemma:zeckenf}
\[
\sum_{k=0}^{\infty} b_{k+1} \beta_k = c = c - f(na) + f(na) = \sum_{k=l+1}^{\infty} c_{k+1} \beta_{k} +\sum_{k=0}^{l} c_{k+1} \beta_{k} = \sum_{k=0}^{\infty} c_{k+1} \beta_k.
\]
Hence by Lemma \ref{lemma:zeckenf} and the uniqueness of the Ostrowski representation, we have that $\sum_{k=0}^{l} b_{k+1} q_k=\sum_{k=0}^{l} c_{k+1} q_k$ is the Ostrowski representation of $n$. The case that $l$ is odd can be shown similarly.
\end{proof}

\begin{lem}\label{lemma:uniquen} Let $l \in \N$. For every $c \in I$ there is a unique $n \in \N_{<q_{l+1}}$ such that
\begin{equation}\label{eq:uniquen}
f(na) + g_{l,1}(na)  \leq c < f(na)  + g_{l,2}(na).
\end{equation}
\end{lem}
\begin{proof} Let $\sum_{k=0}^{\infty} b_{k+1} \beta_k$ be the Ostrowski representation of $c$. We first show the existence of such a natural number $n$. Set $n:=\sum_{k=0}^{l} b_{k+1} q_k$. We will now show that $n$ satisfies \eqref{eq:uniquen}.
Suppose that $l$ is even. Since $l$ is even, $\beta_{l+2k} > 0$ for each $k\in \N$. Then by Fact \ref{fact:osteven}
\[
f(na) - \beta_{l+1} = \sum_{k=0}^{l} b_{k+1} \beta_k + a_{l+3} \beta_{l+2} + a_{l+5} \beta_{l+4} + a_{l+7} \beta_{l+6} + \dots > c.
\]
Suppose that $n \geq q_l$. Then we have $b_{l+1} >0$. Hence by Fact \ref{fact:osteven}
\[
f(na) - (\beta_{l}+\beta_{l+1}) = \sum_{k=0}^{l} b_{k+1} \beta_k + (a_{l+2}-1) \beta_{l+1} + a_{l+4}\beta_{l+3} + a_{l+6} \beta_{l+5} +  \dots \leq c.
\]
Note that the inequality on the right follows immediately from $b_{l+1} > 0$. Now consider that $n < q_l$. Then by Fact \ref{fact:osteven}
\[
f(na) - \beta_{l} = \sum_{k=0}^{l} b_{k+1} \beta_k + a_{l+2} \beta_{l+1} + a_{l+4}\beta_{l+3} + a_{l+6} \beta_{l+5} +  \dots \leq c.
\]
Hence \eqref{eq:uniquen} holds, if $l$ is even. The case that $l$ is odd can be treated similarly. The uniqueness of $n$ follows directly from Lemma \ref{lemma:ostzeck} and the uniqueness of Ostrowski representations.
\end{proof}

\begin{defn} Let $h: V \times I \to \N a$ map a pair $(q_la,c)$ to the unique $na \in \N a_{<q_{l+1}a}$ given by Lemma \ref{lemma:uniquen}.
\end{defn}

\begin{defn} We define
\begin{align*}
E_0 &:= \{ (q_la,c) \in V \times I \ : \ h(q_la,c) < q_la\},\\
E_1 &:= \{ (q_la,c) \in V \times I\ : \ q_la \leq h(q_la,c) < \min \{q_{l+1}a,2q_la\}\}.
\end{align*}
\end{defn}

\begin{lem}\label{lemma:EOs} Let $i\in \{0,1\}$, $l\in \N$, $c\in I$ and let $\sum_{k=0}^{\infty} b_{k+1} \beta_k$ be Ostrowski representation of $c$. Then $b_{l+1}= i$ iff $(q_la,c) \in E_i$.
\end{lem}
\begin{proof} Let $q_la\in V$. Then by Lemma \ref{lemma:ostzeck}, $\sum_{k=0}^l b_{k+1} q_k$ is the Ostrowski representation of $h(q_la,c)/a$. Since the Ostrowski representation of a natural number is obtained by a greedy algorithm, we have $q_l \leq h(q_la,c)/a < \min \{q_{l+1}, 2q_l\}$ iff $b_l = 1$, and $h(q_la,c)/a < q_l$ iff $b_l = 0$. The statement of the Lemma follows immediately.
\end{proof}

\begin{defn} Define $J$ to be set of $c\in I$ such that $(q_la,c) \in E_0 \cup E_1$ for all $q_la \in V$. Let $d \in J$ be the unique element in $J$ such that $(a,d) \in E_1$ and
\[
\forall q_la \in V \ (q_la,d) \in E_1 \textrm{ iff } (q_{l+1}a,d) \notin E_1.
\]
Let $W := \{ q_l \in V \ : \ (q_la,d) \in E_1 \}$.
\end{defn}
It is easy to check that $q_la \in W$ iff $l$ is odd.

\begin{defn} Define $J'$ to be the set of all $c \in J$ such that $(q_la,c) \in E_0$ whenever $q_la \notin W$.
Define $h_1 : W \to \N$ to be the function that maps $q_la$ to $\frac{l-1}{2}$. \newline
Let $h_2: J' \to \Cal P(\N)$ be the function that maps $c \in J'$ to $\{ \frac{l-1}{2} \ : \ (q_la,c)\in E_1 \}$.
\end{defn}

\begin{thm} The map $h=(h_1,h_2) : (W,J',s_{W},E_1) \to (\N,\Cal P(\N),s_{\N},\in)$ is an isomorphism.
\end{thm}
\begin{proof} It follows immediately from the remark after the definition of $W$ that $h_1: (W,s_W) \to (\N,s_{\N})$ is an isomorphism. By definition of $W$, we have that $c \in J$ is in $J'$ if and only if $(q_la,c) \in E_0$ for every even $l\in \N$. Given a subset $X \subseteq \N$, one can easily find a unique $c \in I$ such that
\[
c = \sum_{k \in X} \beta_{2k+1}.
\]
We directly get that $c \in J'$ and for every $k \in \N$, we have $k \in X$ iff $(q_{2k+1}a,c) \in E_1$. Hence $h_2(c) = X$ and $c$ is the unique element in $J'$ with this property. From the construction it follows directly that $(q_la, c) \in E_1$ iff $h_1(q_la) \in h_2(c)$, for every $q_la\in W$ and $c \in J'$.
\end{proof}

Hence $\Cal R_a$ defines an isomorphic copy of $\Cal B$. This proves Theorem D.

\section{Defining multiplication in $\Cal R_{\varphi}$}

\noindent Let $\varphi := \frac{1+\sqrt{5}}{2}$ be the golden ration. In this section it will be shown that multiplication by $\varphi$ is definable in $\Cal R_{\varphi}$. Since the continued fraction expansion of $\varphi$ is $[1;1,\dots]$, we get by Fact \ref{fact:recursive} that $q_k$ is the $k$-th Fibonacci number, while $p_k$ is $k+1$-Fibonacci number. So in particular, $q_{k+1} = p_k$ and $\beta_k = q_k \varphi - q_{k+1}$. Moreover, because of the special form of the continued fraction expansion of $\varphi$, we get that $\zeta_k = \varphi$ for every $k\in \N$. Hence $\beta_{k+1} = -\frac{\beta_k}{\varphi}$ by Fact \ref{fact:beta}. Loosely  speaking, this will allow us to realizes multiplication by $\varphi$ as a shift operation on the Ostrowski representations.\\

\noindent We will use the notation from the previous section. In particular, $f, E_0$ and $E_1$ are as defined before.

\begin{defn} Let $L: \N\varphi \to \N\varphi$ map $n\varphi \in \N\varphi$ to the unique element $m\varphi \in \N \varphi $ such that
$(q_k\varphi,f(m\varphi)) \in E_1$ iff $(q_{k+1}\varphi,f(n\varphi)) \in E_1$ for every $k\geq 1$.\newline
Let $T_1: \N\varphi \to \R$ map $n\varphi$ to
\[
\left\{ \begin{array}{ll}
      L(n\varphi)-f(L(n\varphi))+1, & \hbox{if $(\varphi,f(n\varphi)) \in E_1$;} \\
      L(n\varphi)-f(L(n\varphi)), & \hbox{otherwise.}
    \end{array}
\right.
\]
Let $T_2: \N \varphi \to \R$ map $n\varphi$ to
\[
\left\{ \begin{array}{ll}
      f(L(n\varphi))+\varphi-1, & \hbox{if $(\varphi,f(n\varphi)) \in E_1$;} \\
      f(L(n\varphi)), & \hbox{otherwise.}
    \end{array}
\right.
\]
\end{defn}

\begin{lem}\label{lemma:osshift} Let $n\in \N$ and let $\sum_{k} b_{k+1} q_k$ be the Ostrowski representation of $n$. Then the Ostrowski representation of $\varphi^{-1}L(n\varphi)$ is
\[
\sum_{k} b_{k+2} q_k.
\]
\end{lem}
\begin{proof} Set $m:= \varphi^{-1}L(n\varphi)$ and let $\sum_{k} c_{k+1} q_k$ be the Ostrowski representation on $m$. It is left to show that $c_{k+1} = b_{k+2}$ for every $k\in \N$.
By Lemma \ref{lemma:zeckenf} the Ostrowski representation of $f(m\varphi)$ is $\sum_{k} c_{k+1} \beta_k$ and the Ostrowski representation of $f(n\varphi)$ is $\sum_{k} b_{k+1} \beta_k$. By Lemma \ref{lemma:EOs}, $c_{k+1} = 1$ iff $(q_k\varphi,f(m\varphi)) \in E_1$. By definition of the map $L$, this occurs iff $(q_{k+1}\varphi,f(n\varphi)) \in E_1$. Again, by Lemma \ref{lemma:EOs} this happens if and only if $b_{k+2} = 1$. Hence $c_{k+1} = b_{k+2}$.
\end{proof}

\begin{lem}\label{lemma:T1} Let $n\in \N$. Then $T_1(n\varphi) = n$.
\end{lem}
\begin{proof}
Let $\sum_{k} b_{k+1} q_k$ be the Ostrowski representation of $n$. By Lemma \ref{lemma:osshift} and $\beta_k = q_k \varphi - q_{k+1}$, we get that
\begin{align*}
n &= \sum_{k} b_{k+1} q_k\\
&= (\sum_{k>1} b_{k+1}   q_{k-1})\varphi + (\sum_{k>1} b_{k+1}(q_k - \varphi q_{k-1})) + b_2 q_1\\
&= L(n\varphi) - f(L(n\varphi)) + b_2 \\
&= T_1(n\varphi).
\end{align*}
\end{proof}

\begin{lem}\label{lemma:T2} Let $n \in \N$. Then $\varphi f(n \varphi) = - f(T_2(n\varphi))$.
\end{lem}
\begin{proof} Since $\zeta_k=\varphi$ and $\beta_k = q_k \varphi - q_{k+1}$, we have by Fact \ref{fact:beta} that
\begin{equation}\label{eq:T2}
\frac{q_{k+1} \varphi - q_{k+2}}{q_k \varphi - q_{k+1}} = - \frac{1}{\varphi}.
\end{equation}
Let $\sum_{k} b_{k+1} q_k$ be the Ostrowski representation of $n$. Note that
\[
\varphi \beta_1 = \varphi(\varphi - 2) = \varphi^2-2\varphi = 1-\varphi.
\]
Hence by \eqref{eq:T2}
\begin{align*}
\varphi f(n\varphi) &= \sum_{k} b_{k+1} \varphi(q_k \varphi - q_{k+1}) = - (\sum_{k>0} b_{k+1} (q_{k-1} \varphi - q_{k}))+\varphi b_2\beta_1   \\
&= - f(L(n\varphi)-b_2(\varphi-1) =- T_2(n\varphi).
\end{align*}
\end{proof}

\begin{thm}\label{thm:defm} The function $\lambda_{\varphi} : \R \to \R$ that maps $x \mapsto \varphi x$, is definable in $\Cal R_{\varphi}$.
\end{thm}
\begin{proof} It is enough to define $\lambda_{\varphi}$ on $\R_{\geq 0}$. For $m \in \N$ and $n\varphi \in  \N\varphi$, define a map $P: \N \times \varphi \N \to \R$ by
\[
P(m,n\varphi) := T_1^{-1}(m) - f(T_2(n\varphi)).
\]
This is well-defined, since $T_1$ is injective by Lemma \ref{lemma:T1}, and moreover definable in $\Cal R_{\varphi}$. By Lemma \ref{lemma:T1} and Lemma \ref{lemma:T2}, we have
\[
P(m,f(n\varphi)) = T_1^{-1}(m) - f(T_2(n\varphi)) = \varphi m + \varphi f(n\varphi) = \varphi\cdot (m + f(n\varphi)).
\]
Hence if there are $m,m' \in \N$ and $n\varphi,n'\varphi \in \N\varphi $ with $m + f(n\varphi)=m' + f(n'\varphi)$, we get $P(m,f(n\varphi))=P(m',f(n'\varphi))$. Let $Q : \N + f(\N\varphi) \to  \R$ map $m + f(n\varphi)$ to $P(m,f(n\varphi)$. By the above, $Q$ is well-defined, definable in $\Cal R_{\varphi}$ and $Q(x) = \varphi x$ for all $x \in \N + f(\N\varphi)$. Since $\N + f(\N\varphi)$ is dense in $[1-\varphi,\infty)$ and multiplication by $\varphi$ is continuous, the graph of $\lambda_{\varphi}$ on $[1-\varphi,\infty)$ is the topological closure of the graph of $Q$ in $\R^2$. Hence $\lambda_{\varphi}$ is definable in $\Cal R_{\varphi}$.
\end{proof}

\noindent Theorem B now follows immediately from Theorem A and Theorem \ref{thm:defm}.

\section{Optimality and open questions}

\subsection*{1} Let $\Cal L$ be the language of $\Cal R_a$ for some $a\in \R$. For $a \in \R\setminus \Q$, we have seen that the structure $\Cal R_a$ defines the set $\{ q_ka \ : k > 0\}$, which we denoted by $V$. Since $s_V^{k-1}(a)= q_ka$, it easy to see that for every $k,l \in \N$ there is an $\Cal L$-sentence $\psi_{k,l}$ such that for all $a \in \R \setminus \Q$
\[
\Cal R_a \models \psi_{k,l} \hbox{ iff } q_{k+1} = l q_k + q_{k-1}.
\]
It follows immediately from Fact \ref{fact:recursive} that if $a=[a_0;a_1,\dots]$ and the function that takes $k$ to $a_k$ is non-computable, then the theory of $\Cal R_a$ is undecidable.

\subsection*{2} For $a$ quadratic, quantifier elimination results for $\Cal B$ like \cite[Theorem 1]{Buchi} transfer quite directly to $\Cal R_a$ because of Theorem C and Theorem D. Any attempt of proving substantially different quantifier elimination results for $\Cal R_a$ are likely to fail due to Theorem D.

\subsection*{3} Let $a\in \R\setminus \Q$. By Lemma \ref{lemma:zeckenf}, the function $f: \N a \to \R$ that takes $na \in \N a$ to $\sum_k b_k \beta_k$, where $\sum_{k}b_k q_k$ is the Ostrowski representation of $n$, is definable in $\Cal R_a$. This function maps a closed and discrete set onto a dense subset of the interval $[1-a,2-a)$. Hence together with Theorem A of the current paper, it follows that for $a$ quadratic the structure $\Cal R_a$ satisfies condition (i) of \cite[Theorem A]{HT}, but not its conclusion. Hence condition (ii) can not be dropped from \cite[Theorem A]{HT}.

\subsection*{4} Except for Theorem D not much is known about the structure $\Cal R_a$ when $a$ is not quadratic. For example it is not know whether there is an $a$ such that $\Cal R_a$ defines multiplication on $\R$. Even in the case of Euler's number $e$ we do not know whether the theory of $\Cal R_e$ is decidable or not. Because the continued fraction expansion of $e$ is not periodic, it is unlikely that $\Cal R_e$ can be defined in $\Cal B$, surely not in the way presented here. On the one hand the continued fraction expansion of $e$ is simple enough that other methods might be used to show decidability, but on the other hand the expansion $\Cal S_e$ defines multiplication on $\R$ by \cite[Theorem B]{HT}.

\subsection*{5} It is an open question whether Theorem B holds for all quadratic numbers. However, it seems possible that a closer examination of how multiplication by $a$ interacts with the Ostrowski representations based on $a$, might at least yield decidability of the theory of $\Cal S_a$ for every quadratic $a$.

\subsection*{6} By \cite[Theorem C]{HT}, $(\R,<,+,\Z,a\Z,b\Z)$ defines multiplication on $\R$ and hence every projective set whenever $1,a,b$ are linearly independent over $\Q$. However, we do not know whether there is a set definable in both $(\R,<,+,\Z,a\Z)$ and $(\R,<,+,\Z,b\Z)$ that is not definable in $(\R,<,+,\Z)$. It would be interesting if a version of Cobham's theorem similar to the results in Boigelot, Brusten and Bruy\`ere \cite{BBB} holds in this setting.

\subsection*{7} Let $a \in \R\setminus \Q$. Note that an isomorphic copy of $\Cal R_a$ is definable in the expansion $(\R,<,+,\cdot,e^{\Z},e^{\Z a})$ of the real field, but by \cite[Theorem 1.3]{discrete} the theory of the latter structure is undecidable, even if $a$ is quadratic.

\subsection*{8} Several of the results in this paper can be reformulated to state that certain sets are recognizable by certain automata. For example, Lemma \ref{lem:defoplus123} shows that the graph of addition of real numbers given in Ostrowski representation based on a quadratic irrational is recognizable by a deterministic M\"uller automaton. While not crucial for the main results of this paper, it might still be interesting to give and study explicit constructions of these automata.

\bibliographystyle{plain}
\bibliography{hieronymi}

\end{document}